\documentclass[a4paper,10pt]{article}
\usepackage{amssymb,amsmath,amsthm,amsfonts}
\usepackage{bookmark}
\usepackage{mathrsfs}
\usepackage{multirow}
\usepackage{graphicx}
\usepackage{authblk}
\usepackage{indentfirst}
\usepackage{multicol}
\usepackage{tabu}
\usepackage{url}
\usepackage{fancyhdr}
\usepackage[numbers]{natbib}
\usepackage[all]{xy}
\usepackage{mathtools}
\usepackage[english]{babel}
\usepackage{colortbl}
\usepackage{caption}
\usepackage{hyperref}\usepackage{subfigure}
\usepackage{tikz}\usepackage{float}

\newtheorem{theorem}{Theorem}[section]
\newtheorem{proposition}[theorem]{Proposition}
\newtheorem{lemma}[theorem]{Lemma}
\newtheorem{corollary}[theorem]{Corollary}

\newtheorem{definition}{Definition}[section]
\newtheorem{example}{Example}[section]
\newtheorem{remark}{Remark}[section]
\newcommand{\im}{{\mathrm{im}\hspace{0.1em}}}

\usepackage{xr}
\makeatletter
    \newcommand*{\addFileDependency}[1]{
    \typeout{(#1)}
    \@addtofilelist{#1}
    \IfFileExists{#1}{}{\typeout{No file #1.}}
    }
\makeatother

\title{Interaction homotopy and interaction homology}

\author[1,2]{Jian Liu\thanks{Corresponding author:  liujian2@mail.nankai.edu.cn}}
\author[2]{Dong Chen}
\author[2,4,5]{Guo-Wei Wei }
\affil[1]{Mathematical Science Research Center, Chongqing University of Technology, Chongqing 400054, China}
\affil[2]{Department of Mathematics, Michigan State University, MI, 48824, USA}
\affil[4]{Department of Electrical and Computer Engineering, Michigan State University, MI 48824, USA}
\affil[5]{Department of Biochemistry and Molecular Biology, Michigan State University, MI 48824, USA}

\makeatletter
    \renewcommand*{\@fnsymbol}[1]{\ensuremath{\ifcase#1\or \dagger\or *\or *\or
   \mathsection\or \else\@ctrerr\fi}}
\makeatother
\date{}

\begin{document}
    \maketitle

    \paragraph{Abstract}
  Interactions in complex systems are widely observed across various fields, drawing increased attention from researchers. In mathematics, efforts are made to develop various theories and methods for studying the interactions between spaces. In this work, we present an algebraic topology framework to explore interactions between spaces. We introduce the concept of interaction spaces and investigate their homotopy, singular homology, and simplicial homology. Furthermore, we demonstrate that interaction singular homology serves as an invariant under interaction homotopy. We believe that the proposed framework holds potential for practical applications.


\section{Introduction}\label{section:introduction}

Interactions between spaces play crucial roles in various fields, including biology \cite{barabasi2004network}, physics \cite{boccaletti2006complex}, and computer science \cite{newman2003structure}. Complex networks are commonly used models to study these interactions and have been extensively explored. Efforts are ongoing to build a topological framework that can reveal the essence of these interactions. Typically, interactions among spaces involve their intersections. Several theories, rooted in these intersection relationships, have been developed to characterize the topological structures of complex spaces. The nerve complex was introduced to record the intersection patterns among sets in a family \cite{alexandroff1928allgemeinen,eilenberg2015foundations}. \v{C}ech cohomology is a classical theory to analyze and compute cohomology groups with sheaf coefficients by considering the intersections of open coverings of a topological space \cite{bott1982differential, wells1980differential}.  Intersection homology theory was developed to study stratified spaces by considering allowed cycles with well-behaved intersection properties \cite{goresky1980intersection, kirwan1984cohomology}. Intersection homology has applications in algebraic geometry, differential geometry, and mathematical physics.

Our primary goal is to develop an interaction topology theory that describe the whole interaction system rather than just intersections in the system. We hope to characterize the interactive systems in science and engineering, such as protein-protein interactions and cell-cell interactions.  The concept of interaction cohomology has been featured in various works. For example, interaction cohomology was introduced to explore the topological connectivity of invariant sets in forward or backward self-similar systems \cite{sumi2005dynamics,sumi2009interaction}. More recently, Knill has given a different definition of interaction cohomology  to characterize the Wu characteristic \cite{knill2018cohomology}. In this work, we introduce the concept of interaction spaces to elucidate relationships among different spaces within a complex system. We  focus on exploring the homotopy and homology of these interaction spaces. While intersection homology  studies  stratified spaces, concentrating on the duality of cycles, interaction homology concerns coverings, with a specific emphasis on homotopy invariance. In contrast to \v{C}ech cohomology, which relies on information regarding whether sets in a covering intersect, interaction homology considers the intersections of subspaces within the covering, offering both homotopy and geometric characterizations. Moreover, the interaction cohomology discussed in the references \cite{sumi2005dynamics,sumi2009interaction} investigates self-similar systems, with its cohomology also grounded in \v{C}ech cohomology. Combining these aspects, we explore a framework of interaction homotopy and interaction homology, providing a comprehensive characterization of the topological information within a covering.

An interaction space is a topological space $X$ equipped with a covering $\mathcal{C}=\{X_{i}\}_{1\leq i\leq n}$ of $X$. A morphism between interaction spaces consists a family of continuous maps between spaces that preserve certain intersection relation. We then introduce the homotopy between interaction morphisms and investigate the group structure of the homotopy class. Let $\mathcal{C}_{X}$ and $\mathcal{C}_{Y}$ be two $n$-interaction spaces. It is worth noting that $[S\mathcal{C}_{X},\mathcal{C}_{Y}]$ does not have to be a group, while $[\mathcal{C}_{X},\Omega\mathcal{C}_{Y}]$ is a group. Here, $S\mathcal{C}_{X}$ is the suspension of $\mathcal{C}_{X}$, and $\Omega\mathcal{C}_{Y}$ is the loop space of $\mathcal{C}_{Y}$. This also implies that the natural injection of sets $A:[S\mathcal{C}_{X},\mathcal{C}_{Y}]\to [\mathcal{C}_{X},\Omega\mathcal{C}_{Y}]$ does not have to be a bijection. Furthermore, we introduce the concept of the interaction homotopy group and establish a connection between the interaction homotopy group and the usual homotopy group. Let $\mathcal{C}=\{(Y,y_{0}),(Y,y_{0})\}$ be a based interaction space. Let $\Omega Y$ be the loop space of $Y$ with the based point $\omega_{0}:S^{1}\to y_{0}$. Then there is an isomorphism $\pi_{0}(\Omega\mathcal{C})\cong\pi_{0}(F(\Omega Y-\omega_{0},2))$ between the interaction homotopy group and the homotopy group (Section \ref{subsection:homotopy_group}).

Within the above homotopy, singular interaction homology can also be defined and has been proven to be a functor. Moreover, the singular interaction homology is a homotopy invariant (Theorem \ref{theorem:homotopy_invariant}).
\begin{theorem}
    Let $f,g: \mathcal{C}_{X}\to \mathcal{C}_{Y}$ be morphisms of interaction spaces. If $f$ and $g$ are homotopic, then we have $H_{\ast}(f)=H_{\ast}(g)$.
\end{theorem}
In the simplicial case, we introduce the simplicial interaction homology and the simplicial interaction cohomology. The interaction simplicial maps are exactly the morphisms in the category of interaction spaces. This perspective provides us with a combinatorial view to deal with the interaction homology. The simplicial interaction cohomology introduced here bears similarity to the one defined in \cite{knill2018cohomology}, which is related to the Wu characteristic. Furthermore, a dual isomorphism exists between interaction cohomology and interaction homology over a field (Theorem \ref{theorem:cohomology}).

The paper is organized as follows. In the next section, we will introduce the concept of interaction spaces and focus on the homotopy classes of interaction spaces. In Section \ref{section:singular_homology}, the interaction singular homology is introduced and proved to be a homotopy invariant. Finally, we will introduce the interaction simplicial homology and cohomology to facilitate their applications in point cloud data.

\section{Interaction homotopy}\label{section:interaction}
In this section, we present the basic theory of interaction homotopy and explore its homotopy classes as well as its group structure. Interaction homotopy behaves differently from usual homotopy in many ways. The homotopy properties of interaction spaces depend heavily on how these spaces intersect with each other.

\subsection{Formal homotopy and homotopy}
An \emph{$n$-interaction space} $(X,\mathcal{C})$ is a topological space $X$ equipped with a covering $\mathcal{C}=\{X_{i}\}_{1\leq i\leq n}$ of $X$. A \emph{morphism of $n$-interaction spaces}, denoted by $f=\{f_{i}\}_{1\leq i\leq n}:(X,\{X_{i}\}_{1\leq i\leq n})\to (Y,\{Y_{i}\}_{1\leq i\leq n})$, is a family of continuous maps of topological spaces $f_{i}:X_{i}\to Y_{i},1\leq i\leq n$ satisfying the following condition:
\begin{equation}\label{interaction_condition}
   \text{For any $1\leq i,j\leq n$, $f_{i}(x_{i})=f_{j}(x_{j})$ if and only if $x_{i}=x_{j}$, where $x_{i}\in X_{i},x_{j}\in X_{j}$.}\tag{$\ast$}
\end{equation}
The condition (\ref{interaction_condition}) can be viewed as a structure-preserving property of the morphism of $n$-interaction spaces.
The category of $n$-interaction spaces consists of the interaction spaces equipped with the morphisms of interaction spaces, and we denote this category by $\mathbf{IT_{n}}$.


\begin{remark}
One can construct the category of $n$-interaction spaces, where morphisms are not required to satisfy the condition (\ref{interaction_condition}). In this scenario, homotopy and related concepts can still be defined. Specifically, when considering the case of $n=2$ and $X_{2}\subseteq X_{1}$ for each covering $\mathcal{C}=\{X_{1},X_{2}\}$, this reduces to the relative space.
\end{remark}

A morphism of interaction spaces $f:(X,\{X_{i}\}_{1\leq i\leq n})\to (Y,\{Y_{i}\}_{1\leq i\leq n})$ is an isomorphism if there exists a morphism $g:(Y,\{Y_{i}\}_{1\leq i\leq n})\to (X,\{X_{i}\}_{1\leq i\leq n})$ such that $g\circ f=\mathrm{id}$ and $f\circ g=\mathrm{id}$.
If $f:(X,\{X_{i}\}_{1\leq i\leq n})\to (Y,\{Y_{i}\}_{1\leq i\leq n})$ is a bijection in the category $\mathbf{IT_{n}}$, then it is an isomorphism.

\begin{definition}
    Let $f,g:(X,\{X_{i}\}_{1\leq i\leq n})\to (Y,\{Y_{i}\}_{1\leq i\leq n})$ be morphisms of interaction spaces.

    We say that $f$ and $g$ are \emph{formal homotopic}, denoted by $f \simeq g$, if there exists a morphism of interaction spaces $H:(X\times I,\{X_{i}\times I\}_{1\leq i\leq n})\to (Y,\{Y_{i}\}_{1\leq i\leq n})$ such that $H(x,0)=f$ and $H(x,1)=g$ for all $x\in X$.

    We say $f$ is \emph{homotopic} to $g$ if there are a family of formal homotopies $f\simeq f_{1}\simeq\cdots\simeq  f_{k}\simeq g$ connecting $f$ and $g$, denoted by $f \sim g$.
\end{definition}

\begin{definition}
A morphism of interaction spaces $f:(X,\{X_{i}\}_{1\leq i\leq n})\to (Y,\{Y_{i}\}_{1\leq i\leq n})$ is a \emph{homotopy equivalence} if there exists a morphism of interaction spaces $g:(Y,\{Y_{i}\}_{1\leq i\leq n})\to (X,\{X_{i}\}_{1\leq i\leq n})$ such that $f\circ g\sim \mathrm{id}|_{(Y,\{Y_{i}\}_{1\leq i\leq n})}$ and $g\circ f\sim \mathrm{id}|_{(X,\{X_{i}\}_{1\leq i\leq n})}$.
\end{definition}

\begin{example}
The formal homotopy of interaction morphisms is always reflective and symmetric. But it is not transitive. For example, let $X_{1}=\{x_{1}\},X_{2}=\{x_{2}\}$ be disjoint single point spaces. Let $Y_{1}=Y_{2}=Y$ be the unit disk.
Consider the morphisms of interaction spaces $f=\{f_{1},f_{2}\},g=\{g_{1},g_{2}\}:(X,\{X_{1},X_{2}\})\to (Y,\{Y_{1},Y_{2}\})$, where $f_{1}(x_{1})=a,f_{2}(x_{2})=b,g_{1}(x_{1})=b,g_{2}(x_{2})=c$ for disjoint points $a,b,c\in Y$. If there is a formal homotopy $H=\{H_{1},H_{2}\}:(X\times I,\{X_{1}\times I,X_{2}\times I\})\to (Y,\{Y_{1},Y_{2}\})$ from $f$ to $g$, then one has $H_{1}(x_{1},1)=H_{2}(x_{2},0)$. This contradicts the condition (\ref{interaction_condition}). Note that $h=\{f_{1},g_{2}\}:(X,\{X_{1},X_{2}\})\to (Y,\{Y_{1},Y_{2}\})$ is a morphism of interaction spaces. Then the morphism $F=(F_{1},F_{2}):(X\times I,\{X_{1}\times I,X_{2}\times I\})\to (Y,\{Y_{1},Y_{2}\})$ with $F_{1}(x,t)=f_{1}(x)$ and $F_{2}(x,t)=(1-t)f_{2}(x)+tg_{2}(x)$ gives a formal homotopy from $f$ to $h$. Similarly, the morphism $G=(G_{1},G_{2}):(X\times I,\{X_{1}\times I,X_{2}\times I\})\to (Y,\{Y_{1},Y_{2}\})$ with $G_{1}(x,t)=(1-t)f_{1}(x)+tg_{1}(x)$ and $G_{2}(x,t)=g_{2}(x)$ gives a formal homotopy from $h$ to $g$. This shows that $f$ is homotopic to $g$ rather than formal homotopic to $g$.
\end{example}

Let $f,g:(X,\{X_{i}\}_{1\leq i\leq n})\to (Y,\{Y_{i}\}_{1\leq i\leq n})$ be morphisms of interaction spaces. If $f\simeq g:(X,\{X_{i}\}_{1\leq i\leq n})\to (Y,\{Y_{i}\}_{1\leq i\leq n})$, then one has $f_{i}\simeq g_{i}:X_{i}\to Y_{i}$ for each $1\leq i\leq n$. However, the converse is not always true.
Indeed, the homotopy $H_{i}:X_{i}\times I\to Y_{i}$ from $f_{i}:X_{i}\to Y_{i}$ to $g_{i}:X_{i}\to Y_{i}$ for $1\leq i\leq n$ does not always extend to a morphism of interaction spaces.

\begin{example}
    Let $(X,\{X_{1},X_{2}\})$ and $(Y,\{Y_{1},Y_{2}\})$ be two interaction spaces given by
    \begin{equation*}
      X_{1}=\{(x,y)|x^{2}+y^{2}=1\},X_{2}=\{(0,y)|-1\leq y\leq 1\}, Y_{1}=Y_{2}=\{(x,y)|x^{2}+y^{2}\leq 1\}.
    \end{equation*}
    Here, $X=X_{1}\cup X_{2}$ and $Y=Y_{1}\cup Y_{2}$.
    \begin{figure}[h!]
      \centering
      \begin{tikzpicture}[scale=0.6]
      \begin{scope}
      \draw[thick,blue]  (0,-1)--(0,1);
      \draw[thick,red]  (0,0) circle(1);
      \draw[-latex,thick] (2,1) arc (160:90:3 and 1);
      \filldraw[color=black!30] (7,1.5) circle(1);
      \filldraw[color=black!30] (7,-1.5) circle(1);
      \draw[thick,red]  (7,1.5) circle(1);
      \draw[-latex,thick] (2,-1) arc (200:270:3 and 1);
      \draw[thick,red]  (7,-0.5) arc (90:270:1);
      \draw[thick,blue]  (7,-2.5) arc (270:450:1);
      \draw[thick,blue]  (7,0.5)--(7,2.5);
      \draw (-1.2,0.3) node[above]{$X_{1}$};
      \draw (0,0) node[right]{$X_{2}$};
      \draw (3.2,1.5) node[above]{$f$};
      \draw (3.2,-1.5) node[above]{$g$};
      \end{scope}
      \end{tikzpicture}
    \end{figure}
    Consider the morphisms of interaction spaces $f=\{f_{1},f_{2}\},g=\{g_{1},g_{2}\}$ from $(X,\{X_{1},X_{2}\})$ to $(Y,\{Y_{1},Y_{2}\})$ given by
    \begin{equation*}
      f_{1}(x,y)=(x,y),f_{2}(0,y)=(0,y),g_{1}(x,y)=(-|x|,y),g_{2}(0,y)=(\sqrt{1-y^{2}},y).
    \end{equation*}
    A straightforward verification shows that $f_{1}\simeq g_{1}:X_{1}\to Y_{1}$ and $f_{2}\simeq g_{2}:X_{2}\to Y_{2}$. Suppose there is a formal homotopy $H=\{H_{1},H_{2}\}:(X\times I,\{X_{1}\times I,X_{2}\times I\})\to (Y,\{Y_{1},Y_{2}\})$ from $f$ to $g$ such that $H_{1}(u,0)=f_{1}(u),H_{2}(v,0)=f_{2}(v)$ and $H_{1}(u,1)=g_{1}(u),H_{2}(v,1)=g_{2}(v)$ for any $u\in X_{1},v\in X_{2}$. Note that $g_{2}(0,0)=(1,0)=f_{1}(1,0)$. One has $H_{2}((0,0),1)=H_{1}((1,0),0)$, which does not satisfy the condition (\ref{interaction_condition}). Hence, there is no one-step homotopy from $f$ to $g$. Moreover, we can prove that $f$ is not homotopic to $g$. Indeed, let $(F_{1},G_{1}),\dots,(F_{k},G_{k})$ be a family of one-step homotopies from $f$ to $g$. If there exist elements $x_{1}\in X_{1},x_{2}\in X_{2}\setminus\{(0,-1),(0,1)\}$ and $t\in I$ such that $F_{i}(x_{1},t)=G_{i}(x_{2},t)$ for some $i$, then $(F_{i},G_{i})$ does not satisfy the condition (\ref{interaction_condition}), contradiction. So we always have $F_{i}(x_{1},t)\neq G_{i}(O,t)$ for any $i$ and $t\in I$. Here, $O=(0,0)\in X_{2}$. Consider the continuous map $L_{i}(-,t):X_{1}\to S^{1}$, where $L_{i}(x,t)$ is given by the intersection of the ray $\overrightarrow{G_{i}(O,t)F_{i}(x_{1},t) }$ and the unit circle. Then $L_{1},\dots,L_{k}$ give a one-step homotopy from $L_{1}(-,0)=\mathrm{id}_{S^{1}}$ to $L_{k}(-,1)=g_{1}$. This is a contradiction. Summarily, the morphisms $f$ and $g$ are not homotopic to each other.
\end{example}

\subsection{Homotopy class}
Let $(X,\{X_{i}\}_{1\leq i\leq n})$ and $(Y,\{Y_{i}\}_{1\leq i\leq n})$ be two interaction spaces. We denote
\begin{equation*}
    [(X,\{X_{i}\}_{1\leq i\leq n}),(Y,\{Y_{i}\}_{1\leq i\leq n})]
\end{equation*}
the equivalence class of morphisms from $(X,\{X_{i}\}_{1\leq i\leq n})$ to $(Y,\{Y_{i}\}_{1\leq i\leq n})$ with respect to homotopy.

A \emph{based interaction space} $(X,\{(X_{i},x_{i})\}_{1\leq i\leq n})$ is an interaction space $(X,\{X_{i}\}_{1\leq i\leq n})$ with the based spaces $(X_{i},x_{i})$ for $1\leq i\leq n$. For the sake of simplicity, we use $\{X_{i}\}_{1\leq i\leq n}$ to represent the $n$-interaction space $(X,\{X_{i}\}_{1\leq i\leq n})$, and $\{(X_{i},x_{i})\}_{1\leq i\leq n}$ to represent the based $n$-interaction space $(X,\{(X_{i},x_{i})\}_{1\leq i\leq n})$.

Now, we will focus on the based case. Let $\{(X_{i},x_{i})\}_{1\leq i\leq n}$ be a based interaction space. Let $(S^{1},s_{0})$ be a based circle. The \emph{loop space} $\Omega\{(X_{i},x_{i})\}_{1\leq i\leq n}$ of $\{(X_{i},x_{i})\}_{1\leq i\leq n}$ is the interaction space $\{(X_{i},x_{i})^{(S^{1},s_{0})}\}_{1\leq i\leq n}$, denoted by $\{(\Omega X_{i},\omega_{i})\}_{1\leq i\leq n}$. Here, $\omega_{i}:S^{1}\to X_{i}$ is the constant map given by $\omega(s)=x_{i}$ for all $s\in S^{1}$.
Note that $\Omega (X_{i}\cap X_{j})=\Omega X_{i}\cap \Omega X_{j}$. Thus the construction $\Omega:\mathbf{IT_{n}}\to \mathbf{IT_{n}}$ is functorial. The \emph{suspension} $S\{(X_{i},x_{i})\}_{1\leq i\leq n}$ of $\{(X_{i},x_{i})\}_{1\leq i\leq n}$ is defined by the interaction space $\{(S^{1}\wedge X_{i},\ast)\}_{1\leq i\leq n}$. Here, $\ast$ denotes the corresponding based point for each $i$. Two based spaces have non-empty intersection (as based space) if they have the same based point. Thus the suspension is also functorial since $S(X_{i}\cap X_{j})=SX_{i}\cap SX_{j}$.

\begin{proposition}
    Let $\{(X_{i},x_{i})\}_{1\leq i\leq n}$ and $\{(Y_{i},y_{i})\}_{1\leq i\leq n}$ be based interaction spaces. Suppose that $X_{i}$ is Hausdorff for each $1\leq i\leq n$. Then we have a natural injection of sets
    \begin{equation*}
        A:[\{(SX_{i},\ast)\}_{1\leq i\leq n},\{(Y_{i},y_{i})\}_{1\leq i\leq n}]\to [\{(X_{i},x_{i})\}_{1\leq i\leq n},\{(\Omega Y_{i},\omega_{i})\}_{1\leq i\leq n}].
    \end{equation*}
\end{proposition}
\begin{proof}
    Let $\{f_{i}\}_{1\leq i\leq n}:\{(SX_{i},\ast)\}_{1\leq i\leq n}\to \{(Y_{i},y_{i})\}_{1\leq i\leq n}$ be a morphism of interaction spaces. Note that there is a natural equivalence
    \begin{equation*}
        A_{i}:[(SX_{i},\ast),(Y_{i},y_{i})]\to [(X_{i},x_{i}),(\Omega Y_{i},\omega_{i})]
    \end{equation*}
    for each $i$. Here, $A_{i}[f_{i}]=[\hat{f}_{i}]$ and $(\hat{f}_{i}(x))(s)=f_{i}[s,x]$ for $x\in X_{i},s\in S^{1}$ and $[s,x]\in SX=S^{1}\wedge X$. We will show the map
    \begin{equation*}
        A:[\{(SX_{i},\ast)\}_{1\leq i\leq n},\{(Y_{i},y_{i})\}_{1\leq i\leq n}]\to [\{(X_{i},x_{i})\}_{1\leq i\leq n},\{(\Omega Y_{i},\omega_{i})\}_{1\leq i\leq n}]
    \end{equation*}
    given by $A[\{f_{i}\}_{1\leq i\leq n}]=[\{\hat{f}_{i}\}_{1\leq i\leq n}]$ is well defined. For any $x\in X_{i}\cap X_{j}$ and $s\in S^{1}$, one has $(\hat{f}_{i}(x))(s)=f_{i}[s,x]=f_{j}[s,x]=(\hat{f}_{j}(x))(s)$. On the other hand, if $\hat{f}_{i}(x_{i})=\hat{f}_{j}(x_{j})$ for some $x_{i}\in X_{i},x_{j}\in X_{j}$, then we obtain
    \begin{equation*}
      f_{i}[s,x_{i}]=(\hat{f}_{i}(x_{i}))(s)=(\hat{f}_{j}(x_{j}))(s)=f_{j}[s,x_{j}]
    \end{equation*}
    for any $s\in S^{1}$. Since $\{f_{i}\}_{1\leq i\leq n}:\{(SX_{i},\ast)\}_{1\leq i\leq n}\to \{(Y_{i},y_{i})\}_{1\leq i\leq n}$ is a morphism of interaction spaces, we have $[s,x_{i}]=[s,x_{j}]$ for any $s\in S^{1}$. Consequently, we deduce that $x_{i}=x_{j}$. Thus, the condition (\ref{interaction_condition}) is satisfied. Hence, $\{\hat{f}_{i}\}_{1\leq i\leq n}$ is a morphism of interaction spaces. Suppose $H=\{H_{i}\}_{1\leq i\leq n}:\{(SX_{i}\times I,\ast)\}_{1\leq i\leq n}\to \{(Y_{i},y_{i})\}_{1\leq i\leq n}$ is a formal homotopy from $\{f_{i}\}_{1\leq i\leq n}$ to $\{g_{i}\}_{1\leq i\leq n}$. Let $\hat{H}_{i}(x,t)(s)=H_{i}([s,x],t)$. For any $x\in X_{i}\cap X_{j}$, $t\in I$ and $s\in S^{1}$, one has $\hat{H}_{i}(x,t)(s)=H_{i}([s,x],t)=H_{j}([s,x],t)=\hat{H}_{j}(x,t)(s)$. Besides, if $\hat{H}_{i}(x_{i},t)=\hat{H}_{j}(x_{j},t')$ for some $x_{i}\in X_{i},x_{j}\in X_{j}$ and $t,t'\in I$, we have $H_{i}([s,x_{i}],t)=H_{j}([s,x_{j}],t')$ for any $s\in S^{1}$. Since $H$ is a morphism of interaction spaces, one obtains $[s,x_{i}]=[s,x_{j}]$ and $t=t'$. It follows that $x_{i}=x_{j}$. This implies that $\hat{H}$ is a formal homotopy from $\{\hat{f}_{i}\}_{1\leq i\leq n}$ to $\{\hat{g}_{i}\}_{1\leq i\leq n}$. Moreover, $A$ maps homotopy equivalence to homotopy equivalence through a sequence of formal homotopies.
Thus the map $A$ is well defined. The map $A$ is a natural injection in view of the natural equivalence of $A_{i}$ for $1\leq i\leq n$.
\end{proof}

\subsection{The group structure}
Recall that the loop space $(\Omega Y,\omega)$ is an $H$-group with the multiplication map $\mu:\Omega Y\times \Omega Y\to \Omega Y$ given by
\begin{equation*}
    \mu(\omega,\omega')=\left\{
        \begin{array}{ll}
        \omega(2t), & \hbox{$0\leq t\leq 1/2$;} \\
        \omega'(2t-1), & \hbox{$1/2\leq t\leq 1$.}
        \end{array}
    \right.
\end{equation*}
Here, $\omega,\omega'\in \Omega Y$. Thus we can defined a multiplication map on $\{(\Omega Y_{i},\omega_{i})\}_{1\leq i\leq n}$ by
\begin{equation*}
    \mu:\{(\Omega Y_{i},\omega_{i})\}_{1\leq i\leq n}\times \{(\Omega Y_{i},\omega_{i})\}_{1\leq i\leq n}\to  \{(\Omega Y_{i},\omega_{i})\}_{1\leq i\leq n},
\end{equation*}
where $\mu (\{\alpha_{i}\}_{1\leq i\leq n},\{\beta_{i}\}_{1\leq i\leq n})=\{\mu(\alpha_{i},\beta_{i})\}_{1\leq i\leq n}$ for $\alpha_{i},\beta_{i}\in \Omega Y_{i},i=1,\dots,n$.

\begin{proposition}\label{proposition:group1}
Let $\{(X_{i},x_{i})\}_{1\leq i\leq n}$ and $\{(Y_{i},y_{i})\}_{1\leq i\leq n}$ be based interaction spaces. Then the equivalence class $[\{(X_{i},x_{i})\}_{1\leq i\leq n},\{(\Omega Y_{i},\omega_{i})\}_{1\leq i\leq n}]$ has a group structure with the product $[f]\cdot [g]$ given by the homotopy class of the composition
\begin{equation*}
\begin{split}
    \{X_{i}\}_{1\leq i\leq n}&\stackrel{\Delta}{\longrightarrow} \{X_{i}\}_{1\leq i\leq n}\times \{X_{i}\}_{1\leq i\leq n}\\
    &\stackrel{f\times g}{\longrightarrow}  \{\Omega Y_{i}\}_{1\leq i\leq n}\times \{\Omega Y_{i}\}_{1\leq i\leq n}\stackrel{\mu}{\longrightarrow} \{\Omega Y_{i}\}_{1\leq i\leq n}.
\end{split}
\end{equation*}
Here, $\Delta$ is the diagonal map given by $\Delta(\{a_{i}\}_{1\leq i\leq n})=(\{a_{i}\}_{1\leq i\leq n},\{a_{i}\}_{1\leq i\leq n})$ for $\{a_{i}\}_{1\leq i\leq n}\in \{X_{i}\}_{1\leq i\leq n}$. The identity of the group is the morphism $[\{c_{\omega_{i}}\}_{1\leq i\leq n}]$, where $c_{\omega_{i}}:(X_{i},x_{i})\to (\Omega Y_{i},\omega_{i})$ is a null homotopic map.
\end{proposition}
\begin{proof}
Note that $[\{(X_{i},x_{i})\}_{1\leq i\leq n},\{(\Omega Y_{i},\omega_{i})\}_{1\leq i\leq n}]$ is a subset of the group $\prod\limits_{i=1}^{n}[(X_{i},x_{i}),(\Omega Y_{i},\omega_{i})]$. To show $[\{(X_{i},x_{i})\}_{1\leq i\leq n},\{(\Omega Y_{i},\omega_{i})\}_{1\leq i\leq n}]$ is a group, it suffices to prove that it is closed under group multiplication and inverse.

We will first prove that $\mu\circ (f\times g)\circ\Delta$ is a morphism of interaction spaces. Let $f=\{f_{i}\}_{1\leq i\leq n}$ and $g=\{g_{i}\}_{1\leq i\leq n}$.
Note that $f_{i}(x)=f_{j}(x)$ and $g_{i}(x)=g_{j}(x)$ for any $x\in X_{i}\cap X_{j}$. So we have $\mu(f_{i}(x),g_{i}(x))=\mu(f_{j}(x),g_{j}(x))$ for any $x\in X_{i}\cap X_{j}$. For $x_{i}\in X_{i},x_{j}\in X_{j}$, if $\mu(f_{i}(x_{i}),g_{i}(x_{i}))=\mu(f_{j}(x_{j}),g_{j}(x_{j}))$, we have $(f_{i}(x_{i}))(t)=(f_{j}(x_{j}))(t)$ and $(g_{i}(x_{i}))(t)=(g_{j}(x_{j}))(t)$ for $0\leq t\leq 1$. It follows that $x_{i}=x_{j}$.
Thus the set $[\{(X_{i},x_{i})\}_{1\leq i\leq n},\{(\Omega Y_{i},\omega_{i})\}_{1\leq i\leq n}]$ is closed under group multiplication.

On the other hand, for any morphism $f=\{f_{i}\}_{1\leq i\leq n}\in [\{(X_{i},x_{i})\}_{1\leq i\leq n},\{(\Omega Y_{i},\omega_{i})\}_{1\leq i\leq n}]$, the inverse map of $f$ is given by $\nu\circ f=\{\nu \circ f_{i}\}_{1\leq i\leq n}$. Here, $((\nu\circ f_{i})(x))(t)=(f_{i}(x))(1-t)$ for any $x\in X_{i}$ and $0\leq t\leq 1$.
Indeed, a straightforward verification shows that $\nu\circ f$ satisfies the condition (\ref{interaction_condition}).
Hence, the inverse $[f]^{-1}$ is also in $[\{(X_{i},x_{i})\}_{1\leq i\leq n},\{(\Omega Y_{i},\omega_{i})\}_{1\leq i\leq n}]$.
\end{proof}
For interaction spaces $\{(X_{i},x_{i})\}_{1\leq i\leq n}$ and $\{(Y_{i},y_{i})\}_{1\leq i\leq n}$, the equivalence class
\begin{equation*}
  [\{(SX_{i},\ast)\}_{1\leq i\leq n},\{(Y_{i},y_{i})\}_{1\leq i\leq n}]
\end{equation*}
does not have to be a group. The following example shows that the product is not closed in
\begin{example}\label{example:product}
Let $\{(X_{1},x_{1}),(X_{2},x_{2})\}$ be an interaction space of disjoint circles. Consider the interaction space $\{(Y_{1},y_{1}),(Y_{2},y_{2})\}$ given by
\begin{equation*}
  Y_{1}=\{(x,y)|x^{2}+(y-1)^{2}=1\},Y_{2}=\{(x,y)|x^{2}+(y+1)^{2}=1\}
\end{equation*}
and $y_{1}=(1,1),y_{2}=(1,-1)$. Let $h=\{h_{1},h_{2}\}:\{(X_{1},x_{1}),(X_{2},x_{2})\}\to \{(Y_{1},y_{1}),(Y_{2},y_{2})\}$ be a morphism of interaction spaces.
Note that $[h_{1}],[h_{2}]\in \pi_{1}(S^{1})$. If $h_{1}$ and $h_{2}$ are not null homotopic, they must be surjective maps between circles. Thus there exist points $a_{1}\in X_{1}$ and $a_{2}\in X_{2}$ such that $h_{1}(a_{1})=h_{2}(a_{2})$. This contradicts to the condition (\ref{interaction_condition}). So at least one of $h_{1},h_{2}$ is a null homotopic map.

Let $f=\{\theta_{1},c_{y_{2}}\},g=\{c_{y_{1}},\theta_{2}\}:\{(X_{1},x_{1}),(X_{2},x_{2})\}\to \{(Y_{1},y_{1}),(Y_{2},y_{2})\}$ be morphisms of interaction spaces. Here, $c_{y_{1}},c_{y_{2}}$ are constant maps and $\theta_{1},\theta_{2}$ are homoemorphisms between circles. Suppose $[\{(X_{1},x_{1}),(X_{2},x_{2})\},\{(Y_{1},y_{1}),(Y_{2},y_{2})\}]$ is a group. Assume that $[h]$ is the product of $[f]$ and $[g]$. Since at least one of $h_{1},h_{2}$ is null homotopic, one has $[h]=[f]^{k}$ or $[h]=[g]^{k}$ for some $k\in \mathbb{Z}$.
We may write $[h]=[f]^{k}$. It follows that $[f]^{k}=[f]\circ [g]$. So one has $[g]=[f]^{k-1}$, which is impossible.
\end{example}

Recall that the suspension $(SX,\ast)$ is an $H$-cogroup with the comultiplication $\mu':(SX,\ast)\vee (SX,\ast)\to (SX,\ast)$ given by
\begin{equation*}
      \mu'[t,x]=\left\{
        \begin{array}{ll}
        ([2t,x],\ast), & \hbox{$0\leq t\leq 1/2$;} \\
        (\ast,[2t-1,x]), & \hbox{$1/2\leq t\leq 1$.}
        \end{array}
            \right.
\end{equation*}
We define the comultiplication on $\{(SX_{i},\ast)\}_{1\leq i\leq n}$ as
\begin{equation*}
    \mu':\{(SX_{i},\ast)\}_{1\leq i\leq n}\to \{(SX_{i},\ast)\}_{1\leq i\leq n}\vee \{(SX_{i},\ast)\}_{1\leq i\leq n},
\end{equation*}
where $\mu'(\{[t,z_{i}]\}_{1\leq i\leq n})=\left\{
        \begin{array}{ll}
        \{([2t,z_{i}],\ast)\}_{1\leq i\leq n}, & \hbox{$0\leq t\leq 1/2$;} \\
        \{(\ast,[2t-1,z_{i}])\}_{1\leq i\leq n}, & \hbox{$1/2\leq t\leq 1$.}
        \end{array}
            \right.$

Let $\{(X_{i},x_{i})\}_{1\leq i\leq n}$ and $\{(Y_{i},y_{i})\}_{1\leq i\leq n}$ be based interaction spaces. For any morphisms $f,g:\{(SX_{i},\ast)\}_{1\leq i\leq n}\to \{(Y_{i},y_{i})\}_{1\leq i\leq n}$, we can define a morphism $f\cdot g$ as the composition
\begin{equation*}
\begin{split}
    \{SX_{i}\}_{1\leq i\leq n}&\stackrel{\mu'}{\longrightarrow} \{SX_{i}\}_{1\leq i\leq n}\vee \{SX_{i}\}_{1\leq i\leq n}\\
    &\stackrel{f\vee g}{\longrightarrow}  \{Y_{i}\}_{1\leq i\leq n}\vee \{Y_{i}\}_{1\leq i\leq n}\stackrel{\Delta'}{\longrightarrow} \{Y_{i}\}_{1\leq i\leq n}.
\end{split}
\end{equation*}
Here, $\Delta'$ is the map given by $\Delta'(\{a_{i}\}_{1\leq i\leq n},\{y_{i}\}_{1\leq i\leq n})=\Delta'(\{y_{i}\}_{1\leq i\leq n},\{a_{i}\}_{1\leq i\leq n})=\{a_{i}\}_{1\leq i\leq n}$ for $\{a_{i}\}_{1\leq i\leq n}\in \{Y_{i}\}_{1\leq i\leq n}$.

We can regard $[\{(SX_{i},\ast)\}_{1\leq i\leq n},\{(Y_{i},y_{i})\}_{1\leq i\leq n}]$ as a subset of $\prod\limits_{i=1}^{n}[(SX_{i},\ast),(Y_{i},y_{i})]$.
For any morphism $f=\{f_{i}\}_{1\leq i\leq n}:\{(SX_{i},\ast)\}_{1\leq i\leq n}\to \{(Y_{i},y_{i})\}_{1\leq i\leq n}$, the map $f^{\sharp}=\{f_{i}\circ \nu'_{i}\}_{1\leq i\leq n}$ is a morphism of interaction spaces. Here, $\nu'_{i}:SX_{i}\to SX_{i},[s,x]\mapsto [1-s,x]$. Let $\overline{[\{(SX_{i},\ast)\}_{1\leq i\leq n},\{(Y_{i},y_{i})\}_{1\leq i\leq n}]}$ be the subgroup of $\prod\limits_{i=1}^{n}[(SX_{i},\ast),(Y_{i},y_{i})]$ generated by the set $[\{(SX_{i},\ast)\}_{1\leq i\leq n},\{(Y_{i},y_{i})\}_{1\leq i\leq n}]$. Then we have the following result.

\begin{proposition}\label{proposition:injgroup}
Let $\{(X_{i},x_{i})\}_{1\leq i\leq n}$ and $\{(Y_{i},y_{i})\}_{1\leq i\leq n}$ be based interaction spaces. Suppose that $X_{i}$ is Hausdorff for each $1\leq i\leq n$. Then the natural injection $A$ can extent to a monomorphism of groups
    \begin{equation*}
        \overline{A}:\overline{[\{(SX_{i},\ast)\}_{1\leq i\leq n},\{(Y_{i},y_{i})\}_{1\leq i\leq n}]}\to [\{(X_{i},x_{i})\}_{1\leq i\leq n},\{(\Omega Y_{i},\omega_{i})\}_{1\leq i\leq n}].
    \end{equation*}
\end{proposition}
\begin{proof}
Consider the following commutative diagram.
\begin{equation*}
  \xymatrix{
  [\{(SX_{i},\ast)\}_{1\leq i\leq n},\{(Y_{i},y_{i})\}_{1\leq i\leq n}]\ar@{^{(}->}[r]^{A}\ar@{^{(}->}[d]&[\{(X_{i},x_{i})\}_{1\leq i\leq n},\{(\Omega Y_{i},\omega_{i})\}_{1\leq i\leq n}]\ar@{^{(}->}[d]\\
  \prod\limits_{i=1}^{n}[(SX_{i},\ast),(Y_{i},y_{i})]\ar@{->}[r]^{\prod_{i} A_{i}}_{\cong}&\prod\limits_{i=1}^{n}[(X_{i},x_{i}),(\Omega Y_{i},\omega_{i})]
  }
\end{equation*}
By \cite[Proposition 2.23]{switzer2017algebraic}, the map $\prod_{i} A_{i}$ is an isomorphism of groups. The desired result follows from Proposition \ref{proposition:group1}.
\end{proof}

\subsection{Homotopy group}\label{subsection:homotopy_group}

Now, we will introduce the homotopy group of the interaction space. Let $\{(X_{i},x_{i})\}_{1\leq i\leq n}$ be an interaction space.
For a family of disjoint $k$-spheres $S^{k_{1}},\dots,S^{k_{n}}$, we obtain a family of based spheres $(S^{k_{1}},x_{1}),\dots,(S^{k_{n}},x_{n})$ by gluing $x_{i}$ to the based point of $S^{k_{i}}$. Then we have the interaction sphere $\{(S^{k_{i}},x_{i})\}_{1\leq i\leq n}$.
\begin{definition}
The homotopy group of the interaction space $\{(X_{i},x_{i})\}_{1\leq i\leq n}$ is defined by
\begin{equation*}
  \pi_{k_{1},\dots,k_{n}}(\{(X_{i},x_{i})\}_{1\leq i\leq n})=\overline{[\{(S^{k_{i}},x_{i})\}_{1\leq i\leq n},\{(X_{i},x_{i})\}_{1\leq i\leq n}]}.
\end{equation*}
\end{definition}
If $k_{1}=k_{2}=\cdots=k_{n}=k$, we denote $\pi_{k}(\{(X_{i},x_{i})\}_{1\leq i\leq n})=\pi_{k_{1},\dots,k_{n}}(\{(X_{i},x_{i})\}_{1\leq i\leq n})$. Regarding $\{x_{i}\}_{1\leq i\leq n}$ as an interaction space, by Proposition \ref{proposition:group1}, the homotopy class $\pi_{0}(\Omega^{k}\{(X_{i},x_{i})\}_{1\leq i\leq n})=[\{(S^{0},x_{i})\}_{1\leq i\leq n},\Omega^{k}\{(X_{i},x_{i})\}_{1\leq i\leq n}]$ is a group.
Proposition \ref{proposition:injgroup} shows that $\pi_{k}(\{(X_{i},x_{i})\}_{1\leq i\leq n})$ is a subgroup of $\pi_{0}(\Omega^{k}\{(X_{i},x_{i})\}_{1\leq i\leq n})$. Here, $(S^{0},x_{i}),i=1,\dots,n$ are disjoint $0$-spheres glued to the corresponding based point $x_{i}$.

\begin{example}
Example \ref{example:product} continued. By definition, we have
\begin{equation*}
  \pi_{1}(\{(Y_{1},y_{1}),(Y_{2},y_{2})\})=\overline{[\{(X_{1},x_{1}),(X_{2},x_{2})\},\{(Y_{1},y_{1}),(Y_{2},y_{2})\}]}.
\end{equation*}
Recall that each element in $[\{(X_{1},x_{1}),(X_{2},x_{2})\},\{(Y_{1},y_{1}),(Y_{2},y_{2})\}]$ can be written of the form $[f]^{k}$ or $[g]^{k}$ for some $k\in \mathbb{Z}$. Hence, the group $\pi_{1}(\{(Y_{1},y_{1}),(Y_{2},y_{2})\})$ is generated by $[f]$ and $[g]$. Note that $\pi_{1}(\{(Y_{1},y_{1}),(Y_{2},y_{2})\})$ can be regarded as a subgroup of $\pi_{1}(Y_{1},y_{1})\times \pi_{1}(Y_{2},y_{2})\cong \mathbb{Z}\oplus \mathbb{Z}$. It follows that $\pi_{1}(\{(Y_{1},y_{1}),(Y_{2},y_{2})\})\cong \mathbb{Z}\oplus \mathbb{Z}$. Moreover, one can obtain
\begin{equation*}
  \pi_{1}(\{(Y_{1},y_{1}),(Y_{2},y_{2})\})\cong \pi_{0}(\Omega\{(Y_{1},y_{1}),(Y_{2},y_{2})\}).
\end{equation*}
In this example, the corresponding map $\overline{A}$ is   an isomorphism of groups.
\end{example}

An interaction space $\{(X_{i},x_{i})\}_{1\leq i\leq n}$ is \emph{self-interactive} if all the space $(X_{i},x_{i}),i=1,\dots,n$ are coinciding.
Let $\{(Y,y_{0}),(Y,y_{0})\}$ be a self-interaction space.  Let $\omega_{0}:S^{1}\to y_{0}$ be the based point of $\Omega Y$. The configuration space of $\Omega Y-\omega_{0}$ is given by
\begin{equation*}
  F(\Omega Y-\omega_{0},2)=\{(\omega_{1},\omega_{2})\in \Omega Y\times \Omega Y|~\omega_{0},\omega_{1},\omega_{2}\text{ are all distinct from each other}\}.
\end{equation*}
The following theorem shows the connection between the interaction homotopy group and the usual homotopy group.
\begin{theorem}
$\pi_{0}(\Omega\{(Y,y_{0}),(Y,y_{0})\})\cong\pi_{0}(F(\Omega Y-\omega_{0},2))$.
\end{theorem}
\begin{proof}
By a straightforward calculation, we observe that $[\{(S^{0},x_{0}),(S^{0},x_{0})\},\Omega\{(Y,y_{0}),(Y,y_{0})\}]=F(\Omega Y-\omega_{0},2)/\sim$, where $(\omega_{1},\omega_{2})\sim(\omega_{1}',\omega_{2}')$ if there are a sequence of paths $(\gamma_{1}^{1},\gamma_{2}^{1}),\dots,(\gamma_{1}^{k},\gamma_{2}^{k}):[0,1]\to F(\Omega Y-\omega_{0},2)$ connecting $(\omega_{1},\omega_{2})$ and $(\omega_{1}',\omega_{2}')$ in such a way that $\gamma_{1}^{j}$ and $\gamma_{2}^{j}$ are disjoint for each $j$. To show the desired result, we only need to prove that: if $(\omega_{1},\omega_{2})$ and $(\omega_{1}',\omega_{2}')$ are in the same component of $F(\Omega Y-\omega_{0},2)$, then $(\omega_{1},\omega_{2})\sim(\omega_{1}',\omega_{2}')$.


Suppose $(\omega_{1},\omega_{2})$ and $(\omega_{1}',\omega_{2}')$ are in the same component of $F(\Omega Y-\omega_{0},2)$. Then there is a path $(\gamma_{1},\gamma_{2})$ is a path from $(\omega_{1},\omega_{2})$ to $(\omega_{1}',\omega_{2}')$. We will show that $(\omega_{1},\omega_{2})$ and $(\omega_{1}',\omega_{2}')$ are connected by a sequence of paths with the disjoint property.
For each $y\in \omega(S^{1})$, the set $\omega^{-1}(y)$ is closed. Let $\mathcal{I}(y)=\{[a,b]\subseteq \omega^{-1}(y)|a<b \}$. Note that $\mathcal{I}(y)$ can be an empty set. Let $\tilde{\mathcal{I}}(y)$ be the subset of $\mathcal{I}(y)$ consisting of the maximal intervals in $\mathcal{I}(y)$. Then every two distinct intervals $I_{1},I_{2}$ in $\tilde{\mathcal{I}}(y)$ are disjoint; otherwise, we have $I_{1}\subseteq I_{1}\cup I_{2}$, which contradicts to the maximality of $I_{1}$. Let $\mathcal{I}(\omega)=\bigcup\limits_{y\in Y}\tilde{\mathcal{I}}(y)$. Since every two distinct intervals in $\mathcal{I}(\omega)$ are disjoint, we have
\begin{equation*}
  m(\mathcal{I}(\omega))=\sum\limits_{J\in \mathcal{I}(\omega)}m(J)\leq 1,
\end{equation*}
where $m(-)$ denotes the measure function. This shows that the set $\mathcal{I}(\omega)$ is countable. We denote the intervals in $\mathcal{I}(\omega)$ as $\dots,[a_{j},b_{j}],\dots$ such that $b_{j}<a_{j+1}$ for any $j\in \mathbb{Z}$. Here, we omit the discussion of the cases where $j$ has an upper/lower bound, as the process is essentially similar.
For any positive real numbers $q+r\leq 1$, we define the map $\phi_{q,r}:\Omega Y\to \Omega Y$ by
\begin{equation*}
  \phi_{q,r}(\omega)(s)=\left\{
                    \begin{array}{ll}
                      \omega(a_{j}), & \hbox{$\lambda_{j}\leq s\leq \mu_{j}$;} \\
                      \omega(\mu_{j}+\frac{a_{j+1}-b_{j}}{\lambda_{j+1}-\mu_{j}}(s-\mu_{j})), & \hbox{$\mu_{j}\leq s\leq \lambda_{j+1}$;}\\
                      \omega(s), & \hbox{otherwise}
                    \end{array}
                  \right.
\end{equation*}
Here, $j\in \mathbb{Z}$ and $\lambda_{j}=r_{j}+q(b_{j}-a_{j})$ and $\mu_{j}=a_{j}-(q+r)(b_{j}-a_{j})$. Then one can obtain $m(\phi_{q}(\omega))=qm(\omega)$.

Since $(\omega_{1},\omega_{2})\in F(\Omega Y-\omega_{0},2)$, there exists some $s_{2}\in S^{1}$ such that $\omega_{2}(s_{2})\neq \omega_{1}(s_{2})$. Consider the loop $\omega_{3}\in \Omega Y-\omega_{0}$ given by
\begin{equation*}
  \omega_{3}(s)=\left\{
                  \begin{array}{ll}
                    \omega_{2}(2s), & \hbox{$0\leq s\leq s_{2}/2 $;} \\
                    \omega_{2}(s_{2}), & \hbox{$s_{2}/2\leq s\leq (1+s_{2})/2$;}\\
                    \omega_{2}(2s-1), & \hbox{$(1+s_{2})/2\leq s\leq 1$.}
                  \end{array}
                \right.
\end{equation*}
We set $\alpha_{1}(t)=\omega_{1},0\leq t\leq 1$ and
\begin{equation*}
  \alpha_{2}(t)=\left\{
                  \begin{array}{ll}
                    \omega_{2}(2s), & \hbox{$0\leq s\leq s_{2}(2-t)/2 $;} \\
                    \omega_{2}(s_{2}), & \hbox{$s_{2}(2-t)/2\leq s\leq s_{2}(2-t)/2+t/2$;}\\
                    \omega_{2}(2s-1), & \hbox{$s_{2}(2-t)/2+t/2\leq s\leq 1$.}
                  \end{array}
                \right.
\end{equation*}
Note that the path $\alpha_{1}$ and $\alpha_{2}$ are disjoint.
It follows that $(\alpha_{1},\alpha_{2})$ is a path from $(\omega_{1},\omega_{2})$ to $(\omega_{1},\omega_{3})$ on $F(\Omega Y-\omega_{0},2)$. Thus we have $(\omega_{1},\omega_{2})\sim (\omega_{1},\omega_{3})$.
Similarly, there exists an element $s_{2}'\in S^{1}$ such that $\omega_{2}'(s_{2}')\neq \omega_{1}'(s_{2}')$. And we set
\begin{equation*}
  \omega_{3}'(s)=\left\{
                  \begin{array}{ll}
                    \omega_{2}'(2s), & \hbox{$0\leq s\leq s_{2}'/2 $;} \\
                    \omega_{2}'(s_{2}'), & \hbox{$s_{2}'/2\leq s\leq (1+s_{2}')/2$;}\\
                    \omega_{2}'(2s-1), & \hbox{$(1+s_{2}')/2\leq s\leq 1$.}
                  \end{array}
                \right.
\end{equation*}
Then we have $(\omega_{1}',\omega_{2}')\sim (\omega_{1}',\omega_{3}')$. On the other hand, consider the path $\tilde{\gamma}_{2}:I\to \Omega Y-\omega_{0}$ from $\omega_{3}$ to $\omega_{3}'$ given by
\begin{equation*}
  \tilde{\gamma}_{2}(t)(s)=\left\{
                  \begin{array}{ll}
                    \gamma_{2}(t)(2s), & \hbox{$0\leq s\leq ((1-t)s_{2}+ts_{2}')/2 $;} \\
                    \gamma_{2}(t)((1-t)s_{2}+ts_{2}), & \hbox{$((1-t)s_{2}+ts_{2}')/2\leq s\leq (1+(1-t)s_{2}+ts_{2}')/2$;}\\
                    \gamma_{2}(t)(2s-1), & \hbox{$(1+(1-t)s_{2}+ts_{2}')/2\leq s\leq 1$.}
                  \end{array}
                \right.
\end{equation*}
It follows that $m(\tilde{\gamma}_{2}(t))\geq 1/2$.

Now, let $\beta_{1}(t)=\phi_{(1-2t/3),tr)}(\omega_{1})$. Then $\beta_{1}(t)$ is a path from $\omega_{1}$ to $\phi_{1/3,r}(\omega_{1})$. Similarly, we have that $\beta_{1}'(t)=\phi_{(1-2t/3),tr)}(\omega_{1})$ is a path from $\omega_{1}'$ to $\phi_{1/3,r}(\omega_{1}')$. We can choose suitable $r_{0}\in(0,1/3]$ such that $\phi_{1/3,r_{0}}(\omega_{1})\neq \omega_{3}$ and $\phi_{1/3,r_{0}}(\omega_{1}')\neq \omega_{3}'$. Let $\beta_{2}(t)=\omega_{3}$. Then $(\beta_{1},\beta_{2})$ is a path from $(\omega_{1},\omega_{3})$ to $(\phi_{1/3,r}(\omega_{1}),\omega_{3})$ such that $\beta_{1},\beta_{2}$ are disjoint. So we have $(\omega_{1},\omega_{3})\sim(\phi_{1/3,r_{0}}(\omega_{1}),\omega_{3})$. Similarly, one has $(\omega_{1}',\omega_{3}')\sim(\phi_{1/3,r_{0}}(\omega_{1}'),\omega_{3}')$.
\begin{equation*}
  \xymatrix{
  (\omega_{1},\omega_{2})\ar@{.>}[rr]\ar@{->}[d]_{(\alpha_{1},\alpha_{2})}&&(\omega_{1}',\omega_{2}')\ar@{->}[d]^{(\alpha_{1}',\alpha_{2}')}\\
   (\omega_{1},\omega_{3})\ar@{->}[d]_{(\beta_{1},\beta_{2})}&&(\omega_{1}',\omega_{3}')\ar@{->}[d]^{(\beta_{1}',\beta_{2}')}\\
  (\phi_{1/3,r_{0}}(\omega_{1}),\omega_{3})\ar@{->}[rr]^{(\tilde{\gamma}_{1},\tilde{\gamma}_{2})}&&(\phi_{1/3,r_{0}}(\omega_{1}'),\omega_{3}')
  }
\end{equation*}
Let $\tilde{\gamma}_{1}:I\to \Omega Y-\omega_{0}$ be a path given by $\tilde{\gamma}_{1}(t)=\phi_{1/3,r_{0}}(\gamma_{1}(t))$. Note that $\tilde{\gamma}_{1}(0)=\phi_{1/3,r_{0}}(\omega_{1})$ and $\tilde{\gamma}_{1}(1)=\phi_{1/3,r_{0}}(\omega_{1}')$. It can be verified that $\tilde{\gamma}_{1}$ is a path from $\phi_{1/3,r_{0}}(\omega_{1})$ to $\phi_{1/3,r_{0}}(\omega_{1}')$. Moreover, one has $m(\tilde{\gamma}_{1}(t))=\frac{1}{3}m(\tilde{\gamma}_{1}(t))\leq \frac{1}{3}$. Hence, for any $t_{1},t_{2}\in [0,1]$, we have $m(\tilde{\gamma}_{1}(t_{1}))<m(\tilde{\gamma}_{2}(t_{2}))$. This implies that the paths $\tilde{\gamma}_{1}$ and $\tilde{\gamma}_{2}$ are disjoint. Thus we have $(\omega_{1},\omega_{2})\sim (\omega_{1}',\omega_{2}')$.
\end{proof}

\begin{corollary}
$\pi_{1}(\{(Y,y_{0}),(Y,y_{0})\})\hookrightarrow \pi_{0}(F(\Omega Y-\omega_{0},2))$.
\end{corollary}

\begin{example}
Consider $Y=S^{1}\vee S^{1}$ as the wedge sum of two circles, with their intersection $y_{0}$ as the base point. Let $\{(S^{1},x_{1}),(S^{1},x_{2})\}$ be the interaction space of circles intersecting at $x_{1}=x_{2}$. Then the interaction homotopy space $\pi_{1}(\{(Y,y_{0}),(Y,y_{0})\})=\overline{[\{(S^{1},x_{1}),(S^{1},x_{2})\},\{(Y,y_{0}),(Y,y_{0})\}]}$.
Let $h=\{h_{1},h_{2}\}:\{(S^{1},x_{1}),(S^{1},x_{2})\}\to \{(Y,y_{0}),(Y,y_{0})\}$ be a morphism of interaction spaces.
Note that $[h_{1}],[h_{2}]\in [S^{1},Y]$. If $h_{1}$ and $h_{2}$ are not null homotopic, it implies that they must be mapped onto at least one of the circles within $Y$. Therefore, the map $h$ can take one of the following forms: $(\theta_{1},c_{2}),(c_{1},\theta_{2}),(\theta_{1},\theta_{2})$.  Here, $c_{1}$ and $c_{2}$ are null homotopic maps, while $\theta_{1}$ and $\theta_{2}$ are maps that are homotopic to the identity map $\mathrm{id}:S^{1}\to S^{1}$. One can verify that $(\theta_{1},c_{2})\cdot(c_{1},\theta_{2})=(\theta_{1},\theta_{2})$. It follows that $\pi_{1}(\{(Y,y_{0}),(Y,y_{0})\})=\mathbb{Z}\oplus \mathbb{Z}$. Thus the group $\pi_{0}(F(\Omega Y-\omega_{0},2))$ must has the component direct sum $\mathbb{Z}\oplus \mathbb{Z}$.
\end{example}

\section{Interaction singular homology}\label{section:singular_homology}
Homology is a classical homotopy invariant, playing a fundamental role in the study of topological spaces. In this section, we   introduce the homology of interaction spaces, which we call the singular interaction homology. Moreover, we show that the interaction singular homology is an invariant with respect to interaction homotopy. From now on, the ground ring $R$ is assumed to a principal ideal domain. For $R$-modules $A,B$, we write the tensor product of $A$ and $B$ over $R$ by $A\otimes B=A\otimes_{R}B$ for convenience.

\subsection{Interaction singular chain complex}
Let $X$ be a topological space. Recall that a singular $p$-simplex on $X$ is a continuous map $\sigma:\Delta^{p}\to X$, where $\Delta^{p}=\{\sum\limits_{i=0}^{p}t_{i}e_{i}|0\leq t_{i}\leq 1,\sum t_{i}=1\}$ is spanned by the standard basis $e_{0},\dots,e_{p}$ of $\mathbb{R}^{p+1}$. We have a singular chain complex $S_{\ast}(X)$, whose homology is the singular homology of $X$ \cite{hatcher2002algebraic}.

Let $\{X_{i}\}_{1\leq i\leq n}$ be an interaction space. The \emph{interaction singular $(p_{1},\dots,p_{n})$-simplex} $(\sigma_{1},\dots,\sigma_{n})$ consists of a family of singular simplices $\sigma_{i}:\Delta^{p_{i}}\to X_{i},i=1,\dots,n$ satisfying $\bigcap\limits_{i=1}^{n} \im\sigma_{i}\neq\emptyset$. Note that we have a family of chain complexes $S_{\ast}(X_{i})$ for $i=1,\dots,n$. Then the tensor product $\bigotimes\limits_{i=1}^{n}S_{\ast}(X_{i})$ is also a chain complex with the differential given by
\begin{equation*}
  d(\sigma_{1}\otimes \sigma_{2}\otimes\cdots\otimes \sigma_{n})=\sum\limits_{i=1}^{n}(-1)^{\dim \sigma_{1}+\cdots+\dim \sigma_{i-1}}\sigma_{1}\otimes\cdots \otimes d\sigma_{i}\otimes\cdots\otimes \sigma_{n}.
\end{equation*}
Now, consider the sub $R$-module $T_{\ast}(\{X_{i}\}_{1\leq i\leq n})$ of $\bigotimes\limits_{i=1}^{n}S_{\ast}(X_{i})$ generated by the elements $\sigma_{1}\otimes \sigma_{2}\otimes\cdots\otimes \sigma_{n}$ satisfying $\bigcap\limits_{i=1}^{n} \im \sigma_{i}=\emptyset$. Note that $\bigcap\limits_{i=1}^{n} \im\sigma_{i}=\emptyset$ implies $\left(\bigcap\limits_{i\neq j} \sigma_{i}\right)\cap \im d_{r}\sigma_{j}=\emptyset$. Here, $d_{r}:S_{\dim\sigma_{j}}(X_{j})\to S_{\dim\sigma_{j}-1}(X_{j})$ is the face map for $0\leq r\leq \dim \sigma_{j}$. It follows that $T_{\ast}(\{X_{i}\}_{1\leq i\leq n})$ is a sub chain complex of $\bigotimes\limits_{i=1}^{n}S_{\ast}(X_{i})$. The \emph{interaction singular chain complex} of $\{X_{i}\}_{1\leq i\leq n}$ is defined by
\begin{equation*}
  IS_{\ast}(\{X_{i}\}_{1\leq i\leq n})=\left(\bigotimes\limits_{i=1}^{n}S_{\ast}(X_{i})\right)/T_{\ast}(\{X_{i}\}_{1\leq i\leq n}).
\end{equation*}
The \emph{interaction singular homology} of $\{X_{i}\}_{1\leq i\leq n}$ is defined by
\begin{equation*}
  H_{p}(\{X_{i}\}_{1\leq i\leq n})=H_{p}(IS_{\ast}(\{X_{i}\}_{1\leq i\leq n}),\quad p\geq 0.
\end{equation*}
The rank of $H_{p}(\{X_{i}\}_{1\leq i\leq n})$ is the \emph{interaction Betti number}, and we denote it by $\beta_{p}(\{X_{i}\}_{1\leq i\leq n})$.

Recall that the singular chain construction $S_{\ast}:\mathbf{Top}\to \mathbf{Chain}$ from the category of topological spaces to the category of chain complexes is functorial. Now, we will show the functorial property of the interaction chain construction. Let $f=\{f_{i}\}_{1\leq i\leq n}:\{X_{i}\}_{1\leq i\leq n}\to \{Y_{i}\}_{1\leq i\leq n}$ be a morphism of interaction spaces. Then one has a morphism of chain complexes $f^{\sharp}=\bigotimes\limits_{i=1}^{n}f_{i}^{\sharp}:\bigotimes\limits_{i=1}^{n}S_{\ast}(X_{i})\to \bigotimes\limits_{i=1}^{n}S_{\ast}(Y_{i})$. Here, $f_{i}^{\sharp}:S_{\ast}(X_{i})\to S_{\ast}(Y_{i})$ is the morphism of chain complexes induced by $f_{i}:X_{i}\to Y_{i}$. The following lemma indicates that the construction $T_{\ast}:\mathbf{IT_{n}}\to \mathbf{Chain}$ is functorial.
\begin{lemma}\label{lemma:subset}
    Let $f:\{X_{i}\}_{1\leq i\leq n}\to \{Y_{i}\}_{1\leq i\leq n}$ be a morphism of interaction spaces. Then we have $f^{\sharp}(T_{\ast}(\{X_{i}\}_{1\leq i\leq n}))\subseteq T_{\ast}(\{Y_{i}\}_{1\leq i\leq n})$.
\end{lemma}
\begin{proof}
    Let $f=\{f_{i}\}_{1\leq i\leq n}:\{X_{i}\}_{1\leq i\leq n}\to \{Y_{i}\}_{1\leq i\leq n}$. For any element $\sigma_{1}\otimes \sigma_{2}\otimes\cdots\otimes \sigma_{n}$ satisfying $\bigcap\limits_{i=1}^{n} \im \sigma_{i}=\emptyset$, we have $f^{\sharp}(\sigma_{1}\otimes \sigma_{2}\otimes\cdots\otimes \sigma_{n})=(f_{1}\circ\sigma_{1})\otimes (f_{2}\circ\sigma_{2})\otimes\cdots\otimes (f_{n}\circ\sigma_{n})$. Here, $\sigma_{i}$ is the singular simplex of $X_{i}$ for $i=1,\dots,n$. Suppose that $\bigcap\limits_{i=1}^{n}\im(f_{i}\circ \sigma_{i})\neq \emptyset$. Then there exists points $a_{1},\dots,a_{n}$ in the corresponding standard simplices such that $f_{1}(\sigma_{1}(a_{1}))=\cdots=f_{n}(\sigma_{n}(a_{n}))$. It follows that $\sigma_{1}(a_{1})=\cdots=\sigma_{n}(a_{n})$. This contradicts to the assumption $\bigcap\limits_{i=1}^{n} \im \sigma_{i}=\emptyset$. Thus we have $f^{\sharp}(\sigma_{1}\otimes \sigma_{2}\otimes\cdots\otimes \sigma_{n})\subseteq T_{\ast}(\{Y_{i}\}_{1\leq i\leq n})$.
\end{proof}
\begin{lemma}\label{lemma:commutative}
    The construction $IS_{\ast}:\mathbf{IT_{n}}\to \mathbf{Chain}$ from the category of interaction spaces to the category of chain complexes is functorial.
\end{lemma}
\begin{proof}
    Let $f=\{f_{i}\}_{1\leq i\leq n}:\{X_{i}\}_{1\leq i\leq n}\to \{Y_{i}\}_{1\leq i\leq n}$ be a morphism of interaction spaces. Consider the following diagram.
     \begin{equation*}
         \xymatrix{
         0\ar@{->}[r]&T_{\ast}(\{X_{i}\}_{1\leq i\leq n})\ar@{^{(}->}[r]^{\iota}\ar@{->}[d]^{f^{\sharp}|_{T_{\ast}(\{X_{i}\}_{1\leq i\leq n})}} &\bigotimes\limits_{i=1}^{n}S_{\ast}(X_{i})\ar@{->}[r]^{p}\ar@{->}[d]^{f^{\sharp}} &IS_{\ast}(\{X_{i}\}_{1\leq i\leq n})\ar@{->}[r]\ar@{->}[d]^{\bar{f}} &0\\
         0\ar@{->}[r]&T_{\ast}(\{Y_{i}\}_{1\leq i\leq n})\ar@{^{(}->}[r]^{\iota} &\bigotimes\limits_{i=1}^{n}S_{\ast}(Y_{i})\ar@{->}[r]^{p}&IS_{\ast}(\{Y_{i}\}_{1\leq i\leq n})\ar@{->}[r]&0\\
         }
     \end{equation*}
    Here, $\iota$ is the inclusion, $p$ is the canonical projection, and $\bar{f}:IS_{\ast}(\{X_{i}\}_{1\leq i\leq n})\to IS_{\ast}(\{Y_{i}\}_{1\leq i\leq n})$ is given by $\bar{f}([z])=[f(z)], z\in \bigotimes\limits_{i=1}^{n}S_{\ast}(X_{i})$. By Lemma \ref{lemma:subset}, the above diagram is commutative. Thus we have $p\circ f^{\sharp}=\bar{f}\circ p$. By the functorial property of $S_{\ast}$, the construction $IS_{\ast}:\mathbf{IT_{n}}\to \mathbf{Chain}$ is a functor.
\end{proof}

\subsection{Homotopy invariant}
In the standard framework, homology is known as a homotopy invariant. Now, we will proceed to demonstrate that the interaction homology also exhibits the homotopy invariance.

\begin{theorem}\label{theorem:homotopy_invariant}
    Let $f,g: \{X_{i}\}_{1\leq i\leq n}\to \{Y_{i}\}_{1\leq i\leq n}$ be morphisms of interaction spaces. If $f$ and $g$ are homotopic, then we have $H_{\ast}(f)=H_{\ast}(g)$.
\end{theorem}
\begin{proof}
    Since $f\sim g: \{X_{i}\}_{1\leq i\leq n}\to \{Y_{i}\}_{1\leq i\leq n}$, there are homotopies $F_{i}:X_{i}\times I\to Y_{i}$ from $f_{i}$ to $g_{i}$ for $1\leq i\leq n$. By the proof of \cite[Theorem 2.10]{hatcher2002algebraic}, there is a chain homotopy $g_{i}^{\sharp}-f_{i}^{\sharp}=d \circ P_{i}+P_{i}\circ d$ from $f^{\sharp}:S_{\ast}(X_{i})\to S_{\ast}(Y_{i})$ to $g^{\sharp}:S_{\ast}(X_{i})\to S_{\ast}(Y_{i})$ for each $i$. Here, $P_{i}:S_{\ast}(X_{i})\to S_{\ast+1}(Y_{i})$ is the prim operator given by
    \begin{equation*}
        P_{i}(\sigma)=\sum\limits_{j}(-1)^{j}F_{i}\circ (\sigma\times \mathrm{id})[v_{0},\dots,v_{j},w_{j},\dots,w_{p}]
    \end{equation*}
    for $\sigma\in S_{p}(X_{i})$, where $[v_{0},\dots,v_{p}]=\Delta^{p}\times \{0\}$ and $[w_{0},\dots,w_{p}]=\Delta^{p}\times \{1\}$ are bottoms of the prim $\Delta^{p}\times I$.
    The notation $d$ refers to the differential on the corresponding chain complex if there is no ambiguity.

    Now, we will  prove the chain homotopy
    \begin{equation*}
        \bigotimes\limits_{i=1}^{n}f_{i}^{\sharp}\simeq \bigotimes\limits_{i=1}^{n}g_{i}^{\sharp}:\bigotimes\limits_{i=1}^{n}S_{\ast}(X_{i})\to \bigotimes\limits_{i=1}^{n}S_{\ast}(Y_{i}).
    \end{equation*}
    Note that $\bigotimes\limits_{i=1}^{n}f_{i}^{\sharp}=\tilde{f}_{1}\circ\cdots \circ\tilde{f}_{n}$, where $\tilde{f}_{i}=\mathrm{id}\otimes\cdots\otimes \mathrm{id}\otimes f_{i}^{\sharp}\otimes \mathrm{id}\otimes\cdots\otimes\mathrm{id}$ is a morphism of chain complexes from $S_{\ast}(X_{1})\otimes \cdots \otimes S_{\ast}(X_{i})\otimes S_{\ast}(Y_{i+1})\otimes \cdots \otimes S_{\ast}(Y_{n})$ to $S_{\ast}(X_{1})\otimes \cdots \otimes S_{\ast}(X_{i-1})\otimes S_{\ast}(Y_{i})\otimes \cdots \otimes S_{\ast}(Y_{n})$. Similarly, we set $\bigotimes\limits_{i=1}^{n}g_{i}^{\sharp}=\tilde{g}_{1}\circ\cdots \circ\tilde{g}_{n}$, where $\tilde{g}_{i}=\mathrm{id}\otimes\cdots\otimes \mathrm{id}\otimes g_{i}^{\sharp}\otimes \mathrm{id}\otimes\cdots\otimes\mathrm{id}$.
    Since the composition of chain homotopies is a chain homotopy, we only need to prove the chain homotopy
    \begin{equation*}
      \tilde{f}_{i}\simeq \tilde{g}_{i}.
    \end{equation*}
    By a straightforward calculation, we have
    \begin{equation*}
    \begin{split}
         &\mathrm{id}\otimes\cdots\otimes \mathrm{id}\otimes g_{i}^{\sharp}\otimes \mathrm{id}\otimes\cdots\otimes\mathrm{id}-\mathrm{id}\otimes\cdots\otimes \mathrm{id}\otimes f_{i}^{\sharp}\otimes \mathrm{id}\otimes\cdots\otimes\mathrm{id}\\
         =&d(\mathrm{id}\otimes\cdots\otimes \mathrm{id}\otimes P_{i}\otimes \mathrm{id}\otimes\cdots\otimes\mathrm{id})+(\mathrm{id}\otimes\cdots\otimes \mathrm{id}\otimes P_{i}\otimes \mathrm{id}\otimes\cdots\otimes\mathrm{id})d.
    \end{split}
    \end{equation*}
    Thus we obtain the desired homotopy. Let $h_{i}=\mathrm{id}\otimes\cdots\otimes \mathrm{id}\otimes P_{i}\otimes \mathrm{id}\otimes\cdots\otimes\mathrm{id}$. Note that $h_{i}$ is an $R$-module homomorphism from $S_{\ast}(X_{1})\otimes \cdots \otimes S_{\ast}(X_{i})\otimes S_{\ast}(Y_{i+1})\otimes \cdots \otimes S_{\ast}(Y_{n})$ to $S_{\ast}(X_{1})\otimes \cdots \otimes S_{\ast}(X_{i-1})\otimes S_{\ast+1}(Y_{i})\otimes \cdots \otimes S_{\ast}(Y_{n})$. We have $\tilde{g}_{i}-\tilde{f}_{i}=dh_{i}+h_{i}d$.
    A straightforward calculation shows that
    \begin{equation}\label{equation:homotopy}
        \tilde{g}_{1}\circ\cdots \circ\tilde{g}_{n}-\tilde{f}_{1}\circ\cdots \circ\tilde{f}_{n}=dh+hd,
    \end{equation}
    where $h=\sum\limits_{i=1}^{n}\tilde{g}_{1}\circ \cdots\circ \tilde{g}_{i-1}\circ h_{i}\circ \tilde{f}_{i+1}\circ \cdots\circ \tilde{f}_{n}$. We will prove $h(T_{\ast}(\{X_{i}\}_{1\leq i\leq n}))\subseteq T_{\ast}(\{Y_{i}\}_{1\leq i\leq n})$. Note that
    \begin{equation*}
        \tilde{g}_{1}\circ \cdots\circ \tilde{g}_{i-1}\circ h_{i}\circ \tilde{f}_{i+1}\circ \cdots\circ \tilde{f}_{n}=g_{1}\otimes \cdots g_{i-1}\otimes P_{i}\otimes f_{i+1}\otimes \cdots\otimes f_{n}.
    \end{equation*}
     For any element $\sigma_{1}\otimes \sigma_{2}\otimes\cdots\otimes \sigma_{n}$ satisfying $\bigcap\limits_{r=1}^{n} \im \sigma_{r}=\emptyset$, we assume that $\im g_{1}(\sigma_{1})\cap \cdots \cap \im g_{i-1}(\sigma_{i-1})\cap \im F_{i}(\sigma_{i}\times \mathrm{id}) \cap\im f_{i+1}(\sigma_{i+1})\cap \cdots\cap \im f_{n}(\sigma_{n})\neq \emptyset$. Then one has
     \begin{equation*}
         \left(\bigcap_{r=1}^{n}\im F_{r}(\sigma_{r}\times \mathrm{id})\right)\neq\emptyset.
     \end{equation*}
     Since $\bigcap\limits_{r=1}^{n} \im (\sigma_{r}\times \mathrm{id})=\emptyset$, we obtain $\bigcap\limits_{r=1}^{n}\im F_{r}(\sigma_{r}\times \mathrm{id})=\emptyset$. This contradicts to our assumption. So we have
     \begin{equation*}
         \left(g_{1}\otimes \cdots g_{i-1}\otimes P_{i}\otimes f_{i+1}\otimes \cdots\otimes f_{n}\right)(T_{\ast}(\{X_{i}\}_{1\leq i\leq n})\subseteq T_{\ast}(\{X_{i}\}_{1\leq i\leq n}.
     \end{equation*}
     It follows that $h(T_{\ast}(\{X_{i}\}_{1\leq i\leq n}))\subseteq T_{\ast}(\{Y_{i}\}_{1\leq i\leq n})$.

     By Eq. \ref{equation:homotopy}, we have
     \begin{equation*}
         p\circ\left(\bigotimes\limits_{i=1}^{n}g_{i}^{\sharp}\right)- p\circ\left(\bigotimes\limits_{i=1}^{n}f_{i}^{\sharp}\right)=pdh+phd.
     \end{equation*}
     By the proof of Lemma \ref{lemma:commutative}, one has $p\circ\left(\bigotimes\limits_{i=1}^{n}f_{i}^{\sharp}\right)=\bar{f}\circ p$. Here, $\bar{f}:IS_{\ast}(\{X_{i}\}_{1\leq i\leq n})\to IS_{\ast}(\{Y_{i}\}_{1\leq i\leq n})$ is given by $\bar{f}([z])=[f(z)], z\in \bigotimes\limits_{i=1}^{n}S_{\ast}(X_{i})$. It follows that
     \begin{equation*}
         \bar{g}\circ p-\bar{f}\circ p=pdh+phd.
     \end{equation*}
    Consider the $R$-module homomorphism $r:\bigotimes\limits_{i=1}^{n}S_{\ast}(X_{i})\to T_{\ast}(\{X_{i}\}_{1\leq i\leq n})$ given by
    \begin{equation*}
    r(\sigma_{1}\otimes \sigma_{2}\otimes\cdots\otimes \sigma_{n})=\left\{
        \begin{array}{ll}
        \sigma_{1}\otimes \sigma_{2}\otimes\cdots\otimes \sigma_{n}, & \hbox{$\bigcap\limits_{i=1}^{n} \im \sigma_{i}=\emptyset$;} \\
        0, & \hbox{$\bigcap\limits_{i=1}^{n} \im \sigma_{i}\neq\emptyset$.}
        \end{array}
    \right.
    \end{equation*}
    We have $r|_{T_{\ast}(\{X_{i}\}_{1\leq i\leq n})}=\mathrm{id}|_{T_{\ast}(\{X_{i}\}_{1\leq i\leq n})}$. Note that $d(T_{\ast}(\{X_{i}\}_{1\leq i\leq n}))\subseteq T_{\ast}(\{X_{i}\}_{1\leq i\leq n})$. One has $\im (hdr)\subseteq T_{\ast}(\{X_{i}\}_{1\leq i\leq n})$ and $\im (hrd)\subseteq T_{\ast}(\{X_{i}\}_{1\leq i\leq n})$. It follows that
    \begin{equation*}
        phdr=phrd=0.
    \end{equation*}
    Define an $R$-module homomorphism $s:IS_{\ast}(\{X_{i}\}_{1\leq i\leq n})\to T_{\ast}(\{X_{i}\}_{1\leq i\leq n})$ by $s([z])=z-r(z)$ for $z\in IS_{\ast}(\{X_{i}\}_{1\leq i\leq n})$. This definition is well defined since $z_{1}-r(z_{1})=0$ for $z_{1}\in T_{\ast}(\{X_{i}\}_{1\leq i\leq n})$. Moreover, for any $[z]\in IS_{\ast}(\{X_{i}\}_{1\leq i\leq n})$, we have
    \begin{equation*}
        phds([z])=phdz-phdr(z)=phdz=phdz-phr(dz)=phs([dz]).
    \end{equation*}
    Here, $[dz]=pdz=\bar{d}pz=\bar{d}[z]$ and $\bar{d}$ is the differential on the chain complex $IS_{\ast}(\{X_{i}\}_{1\leq i\leq n})$. It follows that $phds=phs\bar{d}$. On the other hand, a straightforward verification shows $p\circ s=\mathrm{id}$ on $IS_{\ast}(\{X_{i}\}_{1\leq i\leq n})$.
    Thus we obtain
    \begin{equation*}
        \bar{g}-\bar{f}=pdhs+phds=\bar{d}(phs)+(phs)\bar{d}.
    \end{equation*}
    This shows that $\bar{f}\simeq \bar{g}$ is a chain homotopy. The desired result follows.
\end{proof}

\begin{corollary}
Let $f: \{X_{i}\}_{1\leq i\leq n}\to \{Y_{i}\}_{1\leq i\leq n}$ be a homotopy equivalence between interaction spaces. Then $H_{p}(f)$ is an isomorphism for any $p\geq 0$.
\end{corollary}

\subsection{Relative interaction homology}
Let $\{X_{i}\}_{1\leq i\leq n}$ be an interaction space. Consider the sub interaction space $\{A_{i}\}_{1\leq i\leq n}$ of $\{X_{i}\}_{1\leq i\leq n}$, i.e., $A_{i}\subseteq X_{i}$ for $1\leq i\leq n$. Note that $S_{\ast}(\{A_{i}\}_{1\leq i\leq n})$ is a subcomplex of $S_{\ast}(\{X_{i}\}_{1\leq i\leq n})$. The \emph{relative chain complex} $S_{\ast}(\{X_{i}\}_{1\leq i\leq n},\{A_{i}\}_{1\leq i\leq n})$ is defined by
\begin{equation*}
  S_{p}(\{X_{i}\}_{1\leq i\leq n},\{A_{i}\}_{1\leq i\leq n})=S_{p}(\{X_{i}\}_{1\leq i\leq n})/S_{p}(\{A_{i}\}_{1\leq i\leq n})
\end{equation*}
with the differential induced by the quotient.
\begin{definition}
The \emph{relative interaction homology} $H_{\ast}(\{X_{i}\}_{1\leq i\leq n},\{A_{i}\}_{1\leq i\leq n})$ is the homology of the chain complex $S_{\ast}(\{X_{i}\}_{1\leq i\leq n},\{A_{i}\}_{1\leq i\leq n})$.
\end{definition}
We also have a long exact sequence of interaction homology induced by a short exact sequence of interaction chain complexes $0\to S_{p}(\{A_{i}\}_{1\leq i\leq n})\to S_{p}(\{X_{i}\}_{1\leq i\leq n})\to S_{p}(\{X_{i}\}_{1\leq i\leq n},\{A_{i}\}_{1\leq i\leq n})$.
\begin{proposition}
For any pair $(\{X_{i}\}_{1\leq i\leq n},\{A_{i}\}_{1\leq i\leq n})$, there is a long exact sequence
\begin{equation*}
\begin{split}
  \cdots\to H_{p}(\{A_{i}\}_{1\leq i\leq n})\to H_{p}(\{X_{i}\}_{1\leq i\leq n})\to  H_{n}(\{X_{i}\}_{1\leq i\leq n},\{A_{i}\}_{1\leq i\leq n})&\stackrel{\delta}{\to}H_{p-1}(\{A_{i}\}_{1\leq i\leq n})\to \cdots \\
    \cdots& \to H_{0}(\{X_{i}\}_{1\leq i\leq n},\{A_{i}\}_{1\leq i\leq n})\to 0.
\end{split}
\end{equation*}
\end{proposition}
Besides, the relative homology is also a homotopy invariant.
\begin{proposition}
Let $f\sim g: (\{X_{i}\}_{1\leq i\leq n},\{A_{i}\}_{1\leq i\leq n})\to (\{Y_{i}\}_{1\leq i\leq n},\{B_{i}\}_{1\leq i\leq n})$ be a homotopy between morphisms of pairs. Then we have $H_{\ast}(f)=H_{\ast}(g)$.
\end{proposition}

\section{Interaction simplicial homology}\label{section:simplicial_homology}

In this section, we introduce the interaction simplicial homology and interaction simplicial cohomology. The interaction simplicial homology provide us with a discrete perspective for analyzing the topological structures of interaction spaces. Moreover, they enable us to harness the power of interaction homology in topological data analysis.

\subsection{Wu characteristic}
Consider the interaction simplicial complex $(K,\{K_{i}\}_{1\leq i\leq n})$, that is, a simplicial complex $K$ equipped with a family of subcomplex $K_{i}$ such that $K=\bigcup\limits_{i=1}^{n}K_{i}$. There is a family of chain complexes $C_{\ast}(K_{i}), i=1,2,\dots,n$. The tensor product $\bigotimes\limits_{i=1}^{n}C_{\ast}(K_{i})$ is also a chain complex with the differential given by
\begin{equation*}
  d(\sigma_{1}\otimes \sigma_{2}\otimes\cdots\otimes \sigma_{n})=\sum\limits_{i=1}^{n}(-1)^{\dim \sigma_{1}+\cdots+\dim \sigma_{i-1}}\sigma_{1}\otimes\cdots \otimes d\sigma_{i}\otimes\cdots\otimes \sigma_{n}.
\end{equation*}
Consider the sub $R$-module $D_{\ast}(\{K_{i}\}_{1\leq i\leq n})$ of $\bigotimes\limits_{i=1}^{n}C_{\ast}(K_{i})$ generated by the elements $\sigma_{1}\otimes \sigma_{2}\otimes\cdots\otimes \sigma_{n}$ for all $\sigma_{1}\in K_{1},\sigma_{2}\in K_{2},\dots, \sigma_{n}\in K_{n}$ satisfying $\bigcap\limits_{i=1}^{n} \sigma_{i}=\emptyset$. Note that $\bigcap\limits_{i=1}^{n} \sigma_{i}=\emptyset$ implies $\left(\bigcap\limits_{i\neq j} \sigma_{i}\right)\cap d_{r}\sigma_{j}=\emptyset$. Here, $d_{r}:(K_{j})_{\dim \sigma_{j}}\to (K_{j})_{\dim \sigma_{j}-1}$ is the face map for $0\leq r\leq \dim \sigma_{j}$. Therefore, $D_{\ast}(\{K_{i}\}_{1\leq i\leq n})$ is a sub chain complex of $\bigotimes\limits_{i=1}^{n}C_{\ast}(K_{i})$.
The \emph{interaction chain complex} is given by
\begin{equation*}
  IC_{\ast}(\{K_{i}\}_{1\leq i\leq n})=\left(\bigotimes\limits_{i=1}^{n}C_{\ast}(K_{i})\right)/D_{\ast}(\{K_{i}\}_{1\leq i\leq n}).
\end{equation*}
The \emph{interaction simplicial homology} of $\{K_{i}\}_{1\leq i\leq n}$ is defined by
\begin{equation*}
  H_{p}(\{K_{i}\}_{1\leq i\leq n})=H_{p}(IC_{\ast}(\{K_{i}\}_{1\leq i\leq n})),\quad p\geq 0.
\end{equation*}
Consider the case $n=2$ and $K_{1}=K_{2}=K$. The \emph{Wu characteristic} of $K$ is defined by $\omega(K)=\sum\limits_{\sigma\sim \tau}(-1)^{\dim\sigma+\dim\tau}$ \cite{wen1959topologie}. Here, $\sigma\sim\tau$ means $\sigma\cap \tau\neq\emptyset$ for $\sigma,\tau\in K$. In other words, the Wu characteristic count the number of intersections of simplices.
The interaction Betti numbers exhibit a close relationship with the Wu characteristic. This connection bears resemblance to the relationship between the Betti numbers and the Euler characteristic in the usual homology theory. To be more precise, we can express it as $\omega(K)=\sum\limits_{p\geq 0}(-1)^{p}\beta_{2,p}$. Consequently, interaction homology can be seen as an analogue of homology that encodes the Wu characteristic.

\subsection{Interaction simplicial map}
Let $\{K_{i}\}_{1\leq i\leq n}$ and $\{L_{i}\}_{1\leq i\leq n}$ be interaction simplicial complexes.
An \emph{interaction simplicial map} $f=\{f_{i}\}_{1\leq i\leq n}:\{K_{i}\}_{1\leq i\leq n}\to \{L_{i}\}_{1\leq i\leq n}$ is a family of simplicial maps $f_{i}:K_{i}\to L_{i},1\leq i\leq n$ such that for any $1\leq i,j\leq n$, $f_{i}(\sigma_{i})=f_{j}(\sigma_{j})$ if and only if $\sigma_{i}=\sigma_{j}$, where $\sigma_{i}\in K_{i},\sigma_{j}\in K_{j}$.
\begin{lemma}
An interaction simplicial map $f=\{f_{i}\}_{1\leq i\leq n}:\{K_{i}\}_{1\leq i\leq n}\to \{L_{i}\}_{1\leq i\leq n}$ is an interaction morphism of simplicial complexes, and vice versa.
\end{lemma}
\begin{proof}
For any point $x$ in $K_{i}\cap K_{j}$, we can find a simplex $\sigma=\langle v_{0},v_{1},\dots,v_{p}\rangle$ containing $x$ in $K_{i}\cap K_{j}$. It follows that $f_{i}(v_{t})=f_{j}(v_{t})$ for $t=0,1,\dots,p$. Suppose that $x=\sum\limits_{t=0}^{p}a_{t}v_{t}$ for $a_{0},a_{1},\dots,a_{p}\in [0,1]$. It follows that $f_{i}(x)=f_{j}(x)$. On the other hand, suppose $f_{i}(x_{i})=f_{j}(x_{j})$ for some $x_{i}\in K_{i},x_{j}\in K_{j}$. Let $\sigma_{i}$ and $\sigma_{j}$ be the minimal simplicies containing $x_{i}$ and $x_{j}$ in $K_{i}$ and $K_{j}$, respectively. Then $\tau=f_{i}(\sigma_{i})\cap f_{j}(\sigma_{j})\neq \emptyset$ is a simplex in $L_{i}\cap L_{j}$. Let $\tau_{i}=\sigma_{i}\cap f_{i}^{-1}(\tau)\subseteq \sigma_{i}$. Note that $x_{i}\in\tau_{i}$. By the minimality of $\sigma_{i}$, one has $\sigma_{i}=\tau_{i}$. It follows that $\sigma_{i}\subseteq f_{i}^{-1}(\tau)$. Hence, we obtain $f_{i}(\sigma_{i})\subseteq \tau$. Thus one has $f_{i}(\sigma_{i})\subseteq f_{j}(\sigma_{j})$. Similarly, we have $f_{j}(\sigma_{j})\subseteq f_{i}(\sigma_{i})$. It follows that $f_{i}(\sigma_{i})= f_{j}(\sigma_{j})$. So we have $\sigma_{i}=\sigma_{j}$. Let $\sigma_{i}=\langle v_{0},v_{1},\dots,v_{p}\rangle$. Since $f_{i}$ and $f_{j}$ are simplicial maps, one has $f_{i}(v_{r})= f_{j}(v_{r})$ for $0\leq r\leq p$.
Suppose $f_{i}(v_{r})=f_{i}(v_{s})$ for some $0\leq r,s\leq p$. Then one has $f_{i}(v_{r})=f_{j}(v_{s})$, and then we obtain $v_{r}=v_{s}$. Thus $f_{i}(v_{0}),\dots,f_{i}(v_{n})$ are linearly independent. Considering $x_{i},x_{j}\in \sigma_{i}$, we can write $x_{i}=\sum\limits_{t=0}^{p}a_{t}v_{t}$ and $x_{j}=\sum\limits_{t=0}^{p}b_{t}v_{t}$, where $a_{i}$ and $b_{i}$ are real numbers. It follows that
\begin{equation*}
  f_{i}(x_{i})-f_{j}(x_{j})=\sum\limits_{t=0}^{p}(a_{t}-b_{t})f_{i}(v_{t})=0.
\end{equation*}
Thus we have $a_{t}=b_{t}$ for any $0\leq t\leq p$. It follows that $x_{i}=x_{j}$. Therefore, the condition (\ref{interaction_condition}) is satisfied, and $f=\{f_{i}\}_{1\leq i\leq n}:\{K_{i}\}_{1\leq i\leq n}\to \{L_{i}\}_{1\leq i\leq n}$ is an interaction morphism of simplicial complexes. The converse is obvious.
\end{proof}
The above lemma highlights that the interaction simplicial map is essentially the simplicial counterpart of the interaction morphism. Moreover, one can define interaction simplicial maps for abstract simplicial complexes in a similar manner. Based on this idea, the morphism of interaction spaces can be approximated by interaction simplicial maps. Moreover, the simplicial homology is expected to be isomorphic to the singular homology.

\subsection{Interaction cohomology}
Let $\{K_{i}\}_{1\leq i\leq n}$ be a simplicial complex. An \emph{interaction cochain} on $\{K_{i}\}_{1\leq i\leq n}$ is an $R$-module homomorphism $f:C_{\ast}(K_{1})\otimes C_{\ast}(K_{2})\otimes\cdots\otimes C_{\ast}(K_{n})\to R$ such that $f(\sigma_{1}\otimes\cdots \otimes\sigma_{n})=0$ for $\bigcap\limits_{i=1}^{n} \sigma_{i}=\emptyset$.
Here, $\sigma_{1}\in K_{1},\sigma_{2}\in K_{2},\dots, \sigma_{n}\in K_{n}$. By definition, for any linear map $f$, we have
\begin{equation*}
\begin{split}
  d(f)(\sigma_{1}\otimes\cdots \otimes\sigma_{n})=&-(-1)^{\deg f}fd(\sigma_{1}\otimes\cdots \otimes\sigma_{n}) \\
    =&-(-1)^{\deg f}f\left(\sum\limits_{i=1}^{n}(-1)^{\dim \sigma_{1}+\cdots+\dim\sigma_{i-1}}\sigma_{1}\otimes\cdots \otimes d\sigma_{i}\otimes\cdots\otimes\sigma_{n}\right).
\end{split}
\end{equation*}
In view of $(\bigcap\limits_{j\neq i}^{n} \sigma_{j})\cap d_{r}\sigma_{i}=\emptyset$ for $0\leq r\leq \dim\sigma_{i}$, the linear map $d(f)$ is an interaction cochain on $\{K_{i}\}_{1\leq i\leq n}$.
Let $p=\sum\limits_{i=1}^{n}p_{i}$. An interaction cochain of form $f:C_{p_{1}}(K_{1})\otimes C_{p_{2}}(K_{2})\otimes\cdots\otimes C_{p_{n}}(K_{n})\to R$ is an interaction $p$-cochain. Let $IC^{p}(\{K_{i}\}_{1\leq i\leq n})$ be the $R$-module generated by the interaction $p$-cochains on $\{K_{i}\}_{1\leq i\leq n}$. Then one has a cochain complex
\begin{equation*}
  0 \to IC^{0}(\{K_{i}\}_{1\leq i\leq n})\stackrel{d}{\to} IC^{1}(\{K_{i}\}_{1\leq i\leq n})\to \cdots \to IC^{p}(\{K_{i}\}_{1\leq i\leq n})\stackrel{d}{\to} IC^{p+1}(\{K_{i}\}_{1\leq i\leq n})\to \cdots.
\end{equation*}
The cochain complex $IC^{\ast}(\{K_{i}\}_{1\leq i\leq n})$ is called the \emph{interaction cochain complex}. The \emph{interaction cohomology} of $\{K_{i}\}_{1\leq i\leq n}$ is defined by
\begin{equation*}
  H^{p}(\{K_{i}\}_{1\leq i\leq n})=H^{p}(IC^{\ast}(\{K_{i}\}_{1\leq i\leq n})),\quad p\geq 0.
\end{equation*}

\begin{lemma}\label{lemma:cohomology}
Let $R=\mathbb{K}$ be a field. Then we have an isomorphism of cochain complexes
\begin{equation*}
  \omega:IC^{\ast}(\{K_{i}\}_{1\leq i\leq n})\stackrel{\cong}{\to} \mathrm{Hom}_{\mathbb{K}}(IC_{\ast}(\{K_{i}\}_{1\leq i\leq n}),\mathbb{K}),
\end{equation*}
where $\omega(f)(\overline{x})=f(x)$ for $\overline{x}\in IC_{\ast}(\{K_{i}\}_{1\leq i\leq n})$.
\end{lemma}
\begin{proof}
First and foremost, it is important to clarify that the notion $\overline{x}$ represents the image of an element $x$ under the projection $\bigotimes\limits_{i=1}^{n}C_{\ast}(K_{i})\to IC_{\ast}(\{K_{i}\}_{1\leq i\leq n})$.
For any $f\in IC^{\ast}(\{K_{i}\}_{1\leq i\leq n})$, we have $D_{\ast}(\{K_{i}\}_{1\leq i\leq n}\subseteq \ker f$. This induces the map
\begin{equation*}
  \omega(f):IC_{\ast}(\{K_{i}\}_{1\leq i\leq n})\to \bigotimes\limits_{i=1}^{n}C_{\ast}(K_{i})\to\mathbb{K}.
\end{equation*}
Note that the differential of $\mathrm{Hom}_{\mathbb{K}}(IC_{\ast}(\{K_{i}\}_{1\leq i\leq n}),\mathbb{K})$ is given by
\begin{equation*}
  d(g)=-(-1)^{\deg g}gd.
\end{equation*}
A straightforward calculation shows $\omega(d(f))(\overline{x})=d(f)(x)=-(-1)^{\deg f}fdx$ and
\begin{equation*}
  (d\omega(f))(\overline{x})=-(-1)^{\deg \omega(f)}\omega(f)d\overline{x}= -(-1)^{\deg f}\omega(f)\overline{dx}=-(-1)^{\deg f}fdx.
\end{equation*}
Thus the chain map $\omega$ is well defined. Moreover, $\omega$ is obviously an injection.

For any $g\in \mathrm{Hom}_{\mathbb{K}}(IC_{\ast}(\{K_{i}\}_{1\leq i\leq n}),\mathbb{K})$, there is a map $\tilde{g}:IC_{\ast}(\{K_{i}\}_{1\leq i\leq n})\to \mathbb{K}$ given by $\tilde{g}(x)=g(\overline{x})$. One has
\begin{equation*}
  \omega(\tilde{g})(\overline{x})=\tilde{g}(x)=g(\overline{x})
\end{equation*}
Since $\omega$ is injective, we obtain $\omega(\tilde{g})=g$. Thus $\omega$ is a surjection. The desired isomorphism follows.
\end{proof}

Similarly to the case of topological spaces, we can establish the universal coefficient theorem for interaction spaces over a field.
\begin{theorem}\label{theorem:cohomology}
Let $R=\mathbb{K}$ be a field. Then the map
\begin{equation*}
  \psi:H^{\ast}(\{K_{i}\}_{1\leq i\leq n})\stackrel{\cong}{\to} \mathrm{Hom}_{\mathbb{K}}(H_{\ast}(\{K_{i}\}_{1\leq i\leq n}),\mathbb{K})
\end{equation*}
given by $\psi([f])[z]=f(z)$ is an isomorphism.
\end{theorem}
\begin{proof}
By the universal coefficient theorem, one has an isomorphism
\begin{equation*}
  \phi:H^{\ast}(\mathrm{Hom}_{\mathbb{K}}(IC_{\ast}(\{K_{i}\}_{1\leq i\leq n}),\mathbb{K}))\stackrel{\cong}{\to} \mathrm{Hom}_{\mathbb{K}}(H_{\ast}(\{K_{i}\}_{1\leq i\leq n}),\mathbb{K}).
\end{equation*}
By Lemma \ref{lemma:cohomology}, we have another isomorphism
\begin{equation*}
  H(\omega):H^{\ast}(\{K_{i}\}_{1\leq i\leq n})\stackrel{\cong}{\to} H^{\ast}(\mathrm{Hom}_{\mathbb{K}}(IC_{\ast}(\{K_{i}\}_{1\leq i\leq n}),\mathbb{K})).
\end{equation*}
Thus the map $\psi=\phi\circ H(\omega)$ is the desired isomorphism.
\end{proof}

\section{Acknowledgments}
This work was supported in part by NIH grants R01GM126189, R01AI164266, and R35GM148196, National Science Foundation grants DMS2052983 and IIS-1900473,   Michigan State University Research Foundation, and  Bristol-Myers Squibb  65109.

\bibliographystyle{plain}  
\bibliography{Reference}

\end{document}